\documentclass[12pt]{amsart}
\usepackage{amssymb,latexsym}
\usepackage{enumerate}

\makeatletter
\@namedef{subjclassname@2010}{%
  \textup{2010} Mathematics Subject Classification}
\makeatother

\newtheorem{thm}{Theorem}[section]

\newtheorem{lem}[thm]{Lemma}

\newtheorem{proposition}{Proposition}
\newtheorem{corollary}{Corollary}

\theoremstyle{definition}

\newtheorem{rem}[thm]{Remark}

\numberwithin{equation}{section}

\begin{document}


\baselineskip=17pt



\title[]{Small gaps between zeros of twisted L-functions }

\author[]{J. B. Conrey}
\address{American Institute of Mathematics\\ 360 Portage Ave\\ Palo Alto,
CA 94306 USA\\  Department of Mathematics\\ Bristol University\\
Bristol BS8 1SN UK} \email{conrey@aimath.org}

\author[]{H. Iwaniec}
\address{ Department of Mathematics\\ Rutgers University\\
Piscataway, NJ 08903 USA} \email{iwaniec@math.rutgers.edu}

\author[]{K. Soundararajan}
\address{Department of Mathematics\\ Stanford University\\ Stanford, CA 94305
USA}\email{ksound@math.stanford.edu}

\date{}

\begin{abstract}
We use the asymptotic large sieve, developed by the authors, to
prove that if the Generalized Riemann Hypothesis is true,   then
there exist many Dirichlet L-functions that have a pair of
consecutive zeros closer together than 0.37 times their average
spacing.  More generally, we investigate zero spacings within the
family of twists by Dirichlet characters of a fixed L-function and
give precise bounds for small gaps which depend only on the degree
of the L-function.
\end{abstract}



\maketitle

\subsection*{}
Dedicated to Professor Andrzej Schinzel on his 75-th birthday

\section{Introduction}
We prove that if the Generalized Riemann Hypothesis is true, then
there exist many Dirichlet L-functions that have a pair of
consecutive zeros closer together than 0.37 times their average
spacing.  More generally, we investigate zero spacings within the
family of twists by Dirichlet characters of a fixed L-function.

Questions about the vertical spacings between consecutive zeros of L-functions first came into prominence in conjunction with Gauss' class number problem. In the initial efforts to show that the class number $h(-d)$ of the imaginary quadratic field $\mathbf{Q}(\sqrt{-d})$ goes to infinity with $d$ it became clear that there was an interesting connection between this very arithmetic question and the seemingly unrelated problem of zero spacings. In subsequent efforts to
 give effective lower bounds for this class number, the question of zero spacings plays an ever important role.
 Loosely speaking, if one could prove that any L-function has a sufficient number of  consecutive zeros whose
 spacing is smaller than 1/2 of what is expected, then one could disprove the existence of the Landau-Siegel zero.
 For a precise statement of this phenomenon, see the paper [CI] of Conrey and Iwaniec.

 Montgomery's efforts in this direction led to the discovery [Mo] that the zeros of the Riemann zeta-function
are likely to have a vertical distribution that in the scaled limit is identical to that of the angle spacings of eigenvalues of unitary matrices.
This is known to number theorists as the GUE conjecture (for Gaussian Unitary Ensemble). It is currently intractable, so any new information
about the distribution of zeros is noteworthy.  Rudnick and Sarnak [RS] extended this conjecture to any fixed L-function.

Our results give new bounds on small gaps between zeros of families of L-functions. We use the
asymptotic large sieve, developed by us, to prove more accurate results than was previously possible.
 An interesting new qualitative result is that we can prove the existence of Dirichlet L-functions which have gaps smaller than 1/2 of their average spacing! However, we assume GRH in order to do this; in particular, this assumption already rules out the existence of Landau-Siegel zeros.

The reader may argue that an unusually small class number implies,
in certain regions that the Generalized Riemann Hypothesis (GRH)
holds {\it and} that the spacing between zeros is at least as large
as 0.5 times the average spacing. So, why does this current paper
not contradict the existence of such small class numbers? The answer
is that we use GRH in two ways. One is to ensure that $\sum_{\rho}
|A(\rho)|^2$ is the same as  $\sum_{\rho} A(\rho)A(1-\rho)$. For
this use we just need to know that the zeros in our range of
summation are on the critical line.
     The other way we use GRH is to bound sums like $\sum_{n\le x}
\mu(n)\chi(n)$ and to get square root cancellation here. For this
purpose we need GRH to hold for $|t| \ll x$, in particular, GRH
holds on the real axis, too. In the above question the range in
which GRH holds typically does not include the real axis (see for
example [St]  where it is shown that the Epstein zeta-function for a
large value of the parameter $k=\sqrt{|d|}/(2 a)$ has zeros on the
critical line and well-spaced for $|t|\le 2k$ apart from 2 real
zeros near 0 and 1). The upshot is that we cannot conclude anything
about class numbers from our theorem.

\vspace{.1in}

ACKNOWLEDGEMENTS.  The research on the subjects of this paper began
many years ago in connection with other questions about the
distribution of zeros of L-functions, and was supported partly by
various NSF grants and the American Institute of Mathematics. We
acknowledge the current support by the NSF grants DMS-1101774,
DMS-1101574, and DMS-1001068.

 \section{Statement of results}
  Let $L_f(s)$ be a primitive automorphic L-function of degree $d$ and level $N$, and let $L_f(s,\chi)$ be its twist by a primitive Dirichlet character $\chi$.
  In the next section we state the specific assumptions we make for $L_f(s)$ and $L_f(s,\chi)$.
\begin{thm}
Assume the Riemann Hypothesis for $L_f(s,\chi)$ for all primitive $\chi$ of modulus coprime to $N$. Define $\mu_d$ with $0< \mu_d < 1$
to be the unique positive real solution $\mu$ of
$$ \mu+2\int_0^{\mu/d} \left(\frac{\sin \pi v}{\pi v}\right)^2 ~dv =1.$$
Then, for any $\epsilon > 0$ and for all sufficiently large $Q$, there is a $q$ with $(q,N)=1$ and  $Q\le q\le 2Q$ and a primitive $\chi \bmod q$
and a pair of zeros $1/2+i\gamma$ and $1/2+i\gamma'$ of $L_f(s,\chi)$ such that
$|\gamma|, |\gamma'| \le 1$ and
\begin{equation*}
|\gamma-\gamma'| < (\mu_d+\epsilon) \frac {2\pi}{d\log Q}
\end{equation*}
The first few values for $\mu_d$ are approximately $\mu_1=0.366$, $\mu_2=0.519$, $\mu_3=0.611$, $\mu_4=0.674$, and $\mu_5=0.719$. As $d\to \infty$ we have $\mu_d\sim 1-\frac{2}{d}+\frac{4}{d^2} +O(d^{-3})$.
\end{thm}
Since $\frac {2\pi}{d\log Q}$ is the average spacing between consecutive zeros at low heights, our theorem may be rephrased to say that
  there are a pair of zeros that are closer
together than $\mu_d$ times the average spacing.

By a similar method we could show that there are gaps as large as $\lambda_d-\epsilon $ times the average, where $\lambda_d$ is the solution of
$$ \lambda-2\int_0^{\lambda/d} \left(\frac{\sin \pi v}{\pi v}\right)^2 ~dv =1.$$
We find that $\lambda_d=1+2/d+4/d^2+O(1/d^3)$.

We have not attempted to get the best possible results. In particular, it is known that there are slightly better choices
 of the coefficients of the function $H_X$; see [CGG], [MoOd], and [FW]. Also, for large gaps, Hall has invented a method that gives unconditional results.
 [Bre1] and [Bre2] has recently used Hall's method to show that the Riemann zeta-function has gaps between consecutive zeros as large as
 2.766 times the average, and in conjunction with [CSI1] to prove the existence of Dirichlet L-functions with gaps as large as 3.54 times the average.

\section{Basic assumptions}

Let
$$L_f(s)=\sum_{n=1}^\infty\frac{\lambda_f(n)}{n^s}$$
be a primitive, automorphic L-function of degree $d$ and level $N$. By this we mean that either $L_f(s)=\zeta(s)$, the Riemann zeta-function, or else  all of the following assumptions:
the  series for $L_f(s)$ is absolutely convergent for
$\Re s>1$ and  $L_f(s)$ continues to an entire function of order 1;  there are numbers $\epsilon\in \mathbb{C}$,  $Q>0$, and $\mu_j\in \mathbb{C}$ with $\Re \mu_j \ge 0$
such that
$$\Phi_f(s):=(\sqrt{N})^s \prod_{j=1}^d \Gamma_{\mathbb{R}}(s+\mu_j) L_f(s), $$
with $\Gamma_{\mathbb{R}}(s)=\pi^{-s/2}\Gamma(s/2)$, satisfies
$$\Phi_f(1-s)=\epsilon \overline{\Phi_f(1-\overline{s})}.$$
We can also write the functional equation in its asymmetric form:
$$L_f(s)=\epsilon_f X_f(s) \overline{L_f(1-\overline{s})}$$
where
$$X_f(s)= \frac{(\sqrt{N})^{1-s} \prod_{j=1}^d \Gamma_{\mathbb{R}}(1-s+\mu_j)}{(\sqrt{N})^s \prod_{j=1}^d \Gamma_{\mathbb{R}}(s+\mu_j)}.
$$
Also,
$L_f(s)$ has an Euler product, i.e.
$$L_f(s)=\prod_p \sum_{j=0}^\infty \lambda_f(p^j)p^{-js}$$
absolutely convergent for $\Re s>1$.

We assume that for any primitive character $\chi \bmod q$, where $(q,N)=1$,
that
$$L_f(s,\chi):=\sum_{n=1}^\infty \frac{\lambda_f(n)\chi(n)}{n^s}$$
is entire and has a functional equation, i.e. that
$$\Phi(s,\chi):=(q^{d/2}\sqrt{N})^s \prod_{j=1}^d \Gamma_{\mathbb{R}}(s+\mu'_j) L_f(s,\chi) $$
for some $\mu_j'$ satisfies
$$\Phi(s,\chi)=\epsilon_{f,\chi} \overline{\Phi(1-\overline{s},\chi)}$$
for some $\epsilon_{f,\chi}$.  In asymmetric form this is
$$L_f(s,\chi)=\epsilon_{f,\chi}X_f(s,\chi)
 \overline{L_f(1-\overline{s},\chi)}$$
where
$$X_f(s,\chi)= \frac{(q^{d/2}\sqrt{N})^{1-s} \prod_{j=1}^d \Gamma_{\mathbb{R}}(1-s+\mu'_j)}{(q^{d/2}\sqrt{N})^s \prod_{j=1}^d \Gamma_{\mathbb{R}}(s+\mu'_j)}.
$$
Note that
$$\frac{X_f'(s,\chi)}{X_f(s,\chi)} = -d\log q+O(1)$$
uniformly for $s\ll 1$.
  Put
\begin{eqnarray*}
\frac{1}{L_f(s,\chi)}&=&\sum_m \mu_f(m) \chi(m)m^{-s}\\
-\frac{L_f'(s,\chi)}{L_f(s,\chi)}&=&\sum_\ell\Lambda_f(\ell)\chi(\ell) \ell^{-s}.
\end{eqnarray*}
Thus $\mu_f(p^\alpha)=0$ if $\alpha>d$ and $\Lambda_f=\mu_f*\log$ is supported on prime powers. Moreover,
$$\mu_f(p)=-\lambda_f(p), \qquad \qquad \Lambda_f(p)=\lambda_f(p)\log p.$$
We make the following additional basic

{\sc Assumption R.} The Rankin-Selberg series
\begin{eqnarray*}
R(s)=\sum_n|\lambda_f(n)|^2 n^{-s}
\end{eqnarray*}
converges absolutely in $\Re s>1$; it has analytic continuation to $\Re s>1/2$ where it has a pole at $s=1$ of order 1.
Moreover, $R(s)$ has a standard zero-free region so that
$$\sum_{p\le x} |\lambda_f(p)|^2 \frac{\log p}{p} =   \log x +O(1).$$

This completes our description of primitive automorphic L-function of degree $d$ and level $N$.

Throughout this paper we will also assume the Riemann Hypothesis for $L_f(s,\chi)$ that any $\rho_\chi=\beta_\chi+i\gamma_\chi$ with $L_f(\rho_\chi)=0$ and $0< \beta_\chi < 1$ satisfies $\beta_\chi=1/2$.

It follows from standard methods that
the number of zeros of $L_f(s,\chi)$ is
$$N_f(T,\chi):=\#\{ \rho_\chi=\beta_\chi + i \gamma_\chi:0< \gamma_\chi \le T \mbox{ and }  L_f(\rho,\chi)=0 \}\sim \frac{T \log (Nq^d T^d)}{2\pi}$$
as $T\to \infty$. Hence the average spacing between consecutive zeros at a height $T$
is
$$\frac{2\pi}{\log (N q^d  T^d)}.$$
For $T=1$ and $Q<q\le 2Q$ with $Q$ large, this is
$$\sim \frac{2\pi}{d\log Q}.$$

\section{The  method}
The  method  is based on an idea that first appeared in a paper of Julia Mueller [M]  and involves the comparison of two averages.
See also [MoOd].
Let
$$H(s,\chi)=H_X(s,\chi)=\sum_{n\le X} \frac{\mu_f(n) \chi(n)}{n^s}$$
be a partial sum of $L_f(s,\chi)^{-1}$.
Consider
\begin{equation}
\mathcal{M}= \sum_{(q,N)=1} \frac{W(q/Q)}{\phi(q)} \sideset{}{^*}\sum_{\chi \bmod q}
\int_0^1
\big| H_X(1/2+it,\chi)\big|^2 ~dt
\end{equation}
and
\begin{equation}
\mathcal{M}(\alpha)= \sum_{(q,N)=1} \frac{W(q/Q)}{\phi(q)} \sideset{}{^*}\sum_{\chi \bmod q}
\sum_{0\le \gamma_\chi <1} \int_{\max\{0,\gamma_\chi-\alpha\}}^{\min\{1,\gamma_\chi+\alpha\}}
\big| H_X(1/2+it ,\chi)\big|^2 ~dt
\end{equation}
If all pairs of zeros of the same L-function are further apart than $2\alpha$ then necessarily $\mathcal{M}(\alpha) <\mathcal{M}$ since the integration in the $\mathcal{M}(\alpha)$ will be a proper subset of the interval [0,1]. Thus, if $\alpha$ is chosen so large that $\mathcal{M}(\alpha) > \mathcal{M}$ then it must be the case that at least one pair of zeros  are closer together than $2\alpha$.
We will show that for every $\epsilon >0$,
\begin{equation}
\mathcal{M}\bigg(\frac{(\mu_d+\epsilon)}{2} \frac{2\pi}{d \log Q}\bigg) > \mathcal{M}
\end{equation}
for sufficiently large $Q$, from which it follows that there must be an
L-function with modulus $q$ between $Q$ and $2Q$ that has a pair of zeros, of height less than 1, which are closer together than $ (\mu_d+\epsilon) \frac{2\pi}{d\log Q}$.

\section{Asymptotic Large sieve}

The asymptotic large sieve allows us to evaluate, in certain circumstances, the sum
\begin{equation*}S=\sum_{q} \frac{W(q/Q)}{\phi(q)} \sideset{}{^*}\sum_{\chi \bmod q} \sum_{m, n\le X}
a_m b_n \chi(m)  \overline{ \chi(n)} .\end{equation*}
where $X=Q^{2-\eta}$ for arbitrarily small $\eta>0$.
For example, suppose that $a_n$ is a sequence of numbers for which
\begin{equation} \label{eqn:assump} \sum_{n\le X} a_n \chi(n) \ll X^{\epsilon}q^{\epsilon}
\end{equation}
for any $\epsilon >0$ and any primitive Dirichlet character $\chi \bmod q$ and that $\sum_{n\le X}|a_n|^2 \ll X^{\epsilon};$
assume that similar bounds hold for $b_n$.
Under the assumption of GRH, the sequence  $a_n=\mu_f(n)/\sqrt{n}$ is such an example. For such a sequence the
Asymptotic Large Sieve asserts (see [CIS1]) that only the diagonal terms make a significant contribution.
\begin{thm} Suppose that  $X=Q^{2-\eta}$ for some $\eta>0$ and that (\ref{eqn:assump}) holds for the sequences $a_n$ and $b_n$. Then, for any
$\epsilon>0$,
\begin{eqnarray*}
S &=& \sum_{q} \frac{W(q/Q)\phi^*(q)}{\phi(q)}\sum_{n\le X\atop (n,q)=1}
a_n b_n  + O_\epsilon\big(Q^{1-\epsilon} \big)
\end{eqnarray*}
\end{thm}
If we execute the sum over $q$, the above may be rewritten as
\begin{eqnarray*}
S &=&\hat{W}(1)\prod_p\left(1-\frac1{p^2}-\frac{1}{p^3}\right)Q\sum_{n\le X} a_n b_n\frac{\phi(n)}{n}
\prod_{p\mid
n}\left(1-\frac1{p^2}-\frac{1}{p^3}\right)^{-1}\\
&&\qquad+O_\epsilon\big(Q^{1-\epsilon} \big)
\end{eqnarray*}

In the last section we sketch a proof of this result.

A slight generalization allows us to restrict the sum over $q$ to a set coprime with a fixed modulus $N$.
Let
\begin{equation*}S_N=\sum_{(q,N)=1} \frac{W(q/Q)}{\phi(q)} \sideset{}{^*}\sum_{\chi \bmod q} \sum_{m, n\le X}
a_m b_n \chi(m)  \overline{ \chi(n)} .\end{equation*}
Then,
\begin{eqnarray*}
S_N &=& \sum_{(q,N)=1} \frac{W(q/Q)\phi^*(q)}{\phi(q)}\sum_{n\le X\atop (n,q)=1}
a_n b_n  + O_\epsilon\big(Q^{1-\epsilon} \big)
\end{eqnarray*}
or
\begin{eqnarray} \label{eqn:SN}
S_N &=& \hat{W}(1) \prod_p\left(1-\frac1{p^2}-\frac{1}{p^3}\right)Q\sum_{n\le X} a_n b_n\frac{\phi(nN)}{nN}
\prod_{p\mid
nN}\left(1-\frac1{p^2}-\frac{1}{p^3}\right)^{-1}\\
&&\qquad+O_\epsilon\big(Q^{1-\epsilon} \big). \nonumber
\end{eqnarray}

\section{Two Propositions}

Let $g(n)$ be a multiplicative function such that
\begin{eqnarray*}
g(n)=\prod_{p\mid n} g(p) \qquad \mathrm{with} \qquad g(p)=1+O(1/\sqrt{p}).
\end{eqnarray*}
Note that $g(p)$ can take negative values at some small primes.

{\sc Problem A.} Evaluate the sum
\begin{eqnarray*}
A(X)=\sum_{n\le X} |\mu_f(n)|^2 g(n) n^{-1}.
\end{eqnarray*}

{\sc Problem B.} Evaluate the double sum
\begin{eqnarray*}
B(X) = \sum_{mn\le X} \overline{\mu_f}(mn)\mu_f(n) \Lambda_f(m) g(mn) m^{-1-\alpha}n^{-1-\beta}
\end{eqnarray*}
where $\alpha,\beta$ are small complex numbers; $|\alpha|,|\beta|\ll (\log X)^{-1}$.

\subsection{Evaluation of $A(X)$}

It follows from Assumption R that the modified series
$$M(s)=\sum_n |\mu_f(n)|^2 n^{-s}$$
has analytic continuation to $\Re s>1/2$ and it has only a pole at $s=1$ of order exactly 1. Indeed both series have Euler products
\begin{eqnarray*}
R(s)=\prod_p R_p(s), \qquad \qquad  M(s)=\prod_p M_p(s)
\end{eqnarray*}
with $R_p(s)=1+|\lambda_f(p)|^2p^{-s}+\dots$, $M_p(s)=1+|\mu_f(p)|^2 p^{-s}+\dots $, so
$$M_p(s)=R_p(s)+O(p^{-2\sigma}).$$
Twisting $M(s)$ by the multiplicative function $g(m)$ does not change much; precisely
\begin{eqnarray*}
K(s)&=& \sum_n |\mu_f(n)|^2 g(n)n^{-s}\\
&=& \prod_p\left(1+g(p)(M_p(s)-1)\right)=M(s)G(s),
\end{eqnarray*}
say, where
$$G(s)=\prod_p\left(1+(g(p)-1)(M_p(s)-1)/M_p(s)\right)
$$
converges absolutely in $\Re s>1/2$.

If $$M(s)\sim c_f (s-1)^{-1} \qquad \qquad  \quad  (c_f>0)$$
then
$$K(s)\sim c_f c_{fg}(s-1)^{-1}$$
with
$$c_{fg}=G(1)=\prod_p \left(1+(g(p)-1)(M_p(1)-1)/M_p(1)\right). $$
Hence we derive by contour integration

\begin{proposition} For $X\ge 2$ we have
$$A(X)= c_fc_{fg}  \log X +O( 1).$$
\end{proposition}
Hence, by partial summation we derive
\begin{corollary}
For $X\ge 2$ we have
$$\sum_{n\le X} |\mu_f(n)|^2 g(n) n^{-1-\beta} =  c_f c_{fg} \int_1^X x^{-\beta} d(\log x) +O(1).$$
\end{corollary}

\subsection{Evaluation of $B(X)$}
Since $\Lambda_f(m)$ is supported on prime powers we can write
\begin{eqnarray*}
B(X)&=&\sum_{n\le X} |\mu_f(n)|^2 g(n) n^{-1-\beta} \sum_{m\le X/n} \overline{\mu_f}(m)\Lambda_f(m)g(m)m^{-1-\alpha}\\
&&\qquad +O(\log X)
\end{eqnarray*}
where the error term comes from a trivial estimation of the terms with $(m,n)\ne 1$. The inner sum over $m\le Y=X/n$ can be replaced by the sum over primes and $g(p)$ can be replaced by 1 up to the existing error term. We get
$$-\sum_{p\le Y}|\lambda_f(p)|^2 (\log p) p^{-1-\alpha} =- \int_1^Y y^{-\alpha} d(\log y) +O(1).$$
Inserting this into $B(X)$ above we get by Corollary A
\begin{eqnarray*}
B(X)&=& -\int_1^X \big(\sum_{n\le X/y} |\mu_f(n)|^2 g(n) n^{-1-\beta}\big) y^{-\alpha} d(\log y) +O(\log X)\\
&=&-  c_f c_{fg}\iint_{xy<X\atop x,y\ge 1}x^{-\beta}y^{-\alpha} d(\log x)  d(\log y)+O(\log X).
\end{eqnarray*}
Changing the variables of integration we conclude
\begin{proposition}
For $X\ge 2$ we have
$$B(X)=-  c_f c_{fg} F(\alpha \log X, \beta \log X) (\log X)^{2}+O(\log X) $$
where
$$F(a,b)=\iint_{u+v\le 1\atop u,v\ge 0} e^{-au-bv} ~du ~dv.$$
\end{proposition}
\begin{rem}
The arithmetic factors $c_f$ and $c_{fg}$ in the asymptotic formulas for $A(X)$ and $B(X)$ agree, of course!
 \end{rem}
 We further have
 \begin{eqnarray*}
 F(a,b)&=& \frac{a(1-e^{-b})-b(1-e^{-a})}{ab(a-b)}\\
 &=& \sum_{m=1}^\infty \frac{(-1)^m}{m!} \frac{a^{m-1}-b^{m-1}}{a-b}.
 \end{eqnarray*}
Note that
$$F(-i\alpha \log X,0)=\frac{X^{i\alpha}-1-i\alpha \log X}{-\alpha^2\log^2 X}.$$

\section{Proof of theorem}

We evaluate $\mathcal{M}$ using Theorem 2 and find
   that with $X=Q^{2-\eta}$ for some small positive $\eta$ the main term arises only from the diagonal.
   It follows from (\ref{eqn:SN}) that
\begin{eqnarray*}
\sum_{(q,N)=1} \frac{W(q/Q)}{\phi(q)} \sideset{}{^*}\sum_{\chi \bmod q}
 \big| H(1/2+it,\chi)\big|^2
&\sim&\hat{W}(1)  Q \sum_{n\le X} \frac{|\mu_f(n)|^2}{n} \frac {\phi(nN)}{nN}\prod_{p\nmid nN}
\bigg( 1-\frac{1}{p^2}-\frac{1}{p^3}\bigg)\\
&=& c \hat{W}(1) r(N)  Q
\sum_{n\le X} \frac{|\mu_f(n)|^2}{n} g_N(n),
\end{eqnarray*}
say, where
$$c=\prod_p \bigg( 1-\frac{1}{p^2}-\frac{1}{p^3}\bigg),$$
$$r(n)=\frac{\phi(n)}{n}\prod_{p\mid n} \bigg( 1-\frac{1}{p^2}-\frac{1}{p^3}\bigg)^{-1}$$
and
$$g_N(n)= r(nN)/r(N)$$
is multiplicative.
Then, by Proposition 1, this is
\begin{eqnarray*}
  \sim c c_f c_{fg_N} r(N)\hat{W}(1)
  Q \log X
\end{eqnarray*}
uniformly for $|t| \le 1$. Hence, the integration over $t$ is trivial and
$\mathcal{M}$ is asymptotic to this same quantity.

To evaluate $\mathcal{M}(\alpha)$ we first remark that if $\alpha \ll (\log Q)^{-1}$, then
$$\mathcal{M}(\alpha)=\mathcal{M}_1(\alpha)+O(Q\log Q)$$ where
\begin{equation*}
\mathcal{M}_1(\alpha)= \sum_{(q,N)=1} \frac{W(q/Q)}{\phi(q)} \sideset{}{^*}\sum_{\chi \bmod q}
\sum_{0\le \gamma_\chi <1} \int_{ \gamma_\chi-\alpha}^{ \gamma_\chi+\alpha}
\big| H_X(1/2+it ,\chi)\big|^2 ~dt.
\end{equation*}
This is because the difference between the two quantities is
$$\ll \sum_{(q,N)=1} \frac{W(q/Q)}{\phi(q)} \sideset{}{^*}\sum_{\chi \bmod q}
 \left(\int_{0}^{\alpha}+\int_{1-\alpha}^1\right)
 | H_X(1/2+it ,\chi)\big|^2  ~dt
$$
which, by our estimation for $\mathcal{M}$ above is $\ll \alpha Q\log^2 Q\ll Q\log Q$.
To evaluate  $\mathcal{M}_1(\alpha)$, we express the sum over $\gamma_\chi$
as a contour integral
\begin{eqnarray*}
\sum_{0\le \gamma_\chi <1} \int_{\gamma_\chi-\alpha}^{\gamma_\chi+\alpha}
\big| H(1/2+it ,\chi)\big|^2 ~dt&=&
\int_{ -\alpha}^{ +\alpha} \sum_{0\le \gamma_\chi <1}
\big| H(1/2+i\gamma+iu ,\chi)\big|^2 ~du\\
&=& \int_{-\alpha}^\alpha\frac{1}{2 \pi i} \int_{\mathcal C} \frac{L'}{L}(s,\chi)
H(s+iu,\chi)\overline{H}(1-s-iu,\overline{\chi})~ds ~du
\end{eqnarray*}
where $\overline{H}(s)=\overline{H(\overline{s})}$ and
where $\mathcal C$ is the contour which consists of the rectangle with vertical
sides $1/2\pm \delta+it$ with $0\le t\le 1$, where $\delta$ is a small
positive constant. If $L_f(s,\chi)$ has a zero on a horizontal edge
(either at $s=1/2$ or $s=1/2+i $)
for some $\chi$, it causes no problem  to slightly adjust the contour to include
these zeros in the interior.

Let's write
$$C_u(s,\chi):=\frac{L'}{L}(s,\chi)H(s+iu,\chi)=\sum_{n=1}^\infty\frac{b_u(n)\chi(n)}{n^s}$$
and we consider
\begin{eqnarray*}
\mathcal{M}_R(u,s):=\sum_{(q,N)=1} \frac{W(q/Q)}{\phi(q)} \sideset{}{^*}\sum_{\chi \bmod q}
\overline{H}(1-s-iu,\overline{\chi})C_u(s,\chi)
\end{eqnarray*}
for $s$ on the right vertical side of the contour $\mathcal C$.
As in the evaluation of $\mathcal{M}$, the main terms arise only from  the diagonal. Thus,
\begin{eqnarray*}
\mathcal{M}_R(u,s)&\sim&\hat{W}(1)  Q \sum_{n\le X} \frac{\overline{\mu_f (n)}}{n^{1-s-iu}}
\frac{b_u(n)}{n^{s}} \frac {\phi(nN)}{nN}\prod_{p\nmid nN}
\bigg( 1-\frac{1}{p^2}-\frac{1}{p^3}\bigg) \\
&= & \hat{W}(1)  Q \sum_{n\le X} \frac{\overline{\mu_f (n)}b_u(n)}{n^{1 -iu}}
  \frac {\phi(nN)}{nN}\prod_{p\nmid nN}
\bigg( 1-\frac{1}{p^2}-\frac{1}{p^3}\bigg)\\
&=& -c\hat{W}(1) r(N) Q \sum_{mn\le X} \frac{\overline{\mu_f (mn)}\Lambda_f(m)\mu_f(n)}{(mn)^{1 -iu}n^{iu}}g_N(mn)
\end{eqnarray*}
where $c$, $r(N)$, and $g_N$ are as above.
The sum over $m$ and $n$ is just $B(X)$ from Proposition 2 with $\alpha=-iu$ and $\beta=0$.
Thus, for $u\ll (\log Q)^{-1}$ we have
\begin{eqnarray*}
\mathcal{M}_R(u,s)&=& c\hat{W}(1)r(N) Q  c_f c_{fg_N} F(-iu \log X, 0) (\log X)^{2}+O(\log X) .
\end{eqnarray*}

Now we consider what happens for the integral over the left side of the rectangle.
Here we let $s\to 1-s$ and use the functional equation
$$\frac{L_f'}{L_f}(1-s,\chi)=\frac{X_f'}{X_f}(s,\chi)- \frac{L_f'}{L_f}(s,\overline{\chi})$$
where $X_f(s,\chi)$ is the factor from the functional equation $L_f(s,\chi)=X_f(s,\chi)L_f(1-s,
\overline{\chi}).$
Thus, we consider
\begin{eqnarray*}
- \sum_{q} \frac{W(q/Q)}{\phi(q)} \sideset{}{^*}\sum_{\chi \bmod q}
\bigg(\frac{X_f'}{X_f}(s,\chi)-\frac{L_f'}{L_f}(s,\overline{\chi})\bigg)
H(1-s +iu, \chi )\overline{H}(s-iu,\overline{\chi})
\end{eqnarray*}
for $s=1/2+\delta-it$ with $0\le t\le 1$;
the minus sign enters because of the change of variable $s\to 1-s$.
Now,
\begin{eqnarray*} \frac{X_f'}{X_f}(s,\chi)=-d\log Q +O(1)
\end{eqnarray*}
uniformly for $|t|\ll 1$ and $Q\le q \le 2Q$. Consequently, the contribution
from the $X_f'/X_f$ term is asymptotically $ d\frac{\mathcal{M} \log Q}{2\pi}$ where $\mathcal{M}$
was the mean we evaluated before.
 Then, we find that
\begin{eqnarray*}
\mathcal{M}_L(u,s):= \sum_{q} \frac{W(q/Q)}{\phi(q)} \sideset{}{^*}\sum_{\chi \bmod q}
H(1-s+iu, \chi)\frac{L_f'}{L_f}(s,\overline{\chi})
\overline{ H}(s-iu,\overline{\chi})=\overline{\mathcal{M}_R(u,s)}.
\end{eqnarray*}
Summarizing, we have that
\begin{eqnarray*}
\mathcal{M}(\alpha)&\sim& \frac{c\hat{W}(1)r(N) Q  c_f c_{fg_N}}{2\pi}\int_{-\alpha}^{\alpha}(d\log Q \log X +2\Re  F(iu \log X, 0)\log^2X )~du
\end{eqnarray*}
compared with
\begin{eqnarray*}
\mathcal{M} &\sim&  c\hat{W}(1)r(N) Q  c_f c_{fg_N}\log X.
\end{eqnarray*}
Thus, $\mathcal{M}(\alpha)> \mathcal{M}$ when $\alpha $ is chosen so large that $h(\alpha)>1$ where
\begin{eqnarray*}
h(\alpha):=
\frac{1}{2\pi}\int_{-\alpha}^{\alpha}
\bigg(d\log Q +2 \Re F(i u \log X,0) \log X
\bigg)~du  .
\end{eqnarray*}
We see that
\begin{eqnarray*}
h(\alpha)&=&  \int_{-\alpha}^\alpha \frac{1}{2\pi}
\bigg(d\log Q +2 \Re  \frac{1+i u\log X -X^{iu}}{u^2}\log X
\bigg)~du\\
&=& \frac{d\alpha \log Q}{\pi} +\frac{4 \log X}{\pi} \int_0^\alpha
\sin^2(\frac{u}{2}\log X)\frac{du}{u^2}
\end{eqnarray*}
Recalling that $X=Q^{2-\eta}$, we see that
\begin{eqnarray*}
h\bigg(\frac{  \pi \mu}{d \log Q} \bigg)&=&\mu +
\frac{4 \log X}{\pi} \int_0^{\pi \mu /(d\log Q)}
\sin^2(\frac{u}{2}\log X)\frac{du}{u^2}\\
&=& \mu + \frac{2}{1-\frac{\eta}{2}}\int_0^{\mu/d} \bigg(\frac{\sin  \pi v(1-\eta/2)}{\pi v }\bigg)^2 ~dv
\end{eqnarray*}
We let
$$j_d(\mu)= \mu + 2\int_0^{\mu/d} \bigg(\frac{\sin  \pi v }
{\pi v }\bigg)^2 ~dv.
$$
Then $\mu_d $ is defined implicitly by $j_d(\mu_d)=1.$  Given an $\epsilon>0$ we can choose $\eta>0$ sufficiently
small so that
\begin{eqnarray*}
h\bigg(\frac{  \pi (\mu_d+\epsilon)}{d \log Q} \bigg)
>1.
\end{eqnarray*}
This proves the theorem.

Remark. We could similarly determine large gaps between consecutive zeros of
$L_f(s,\chi)$.  Using $a_n=1/\sqrt{n}$, an argument similar to the one above leads to
$$j_d^+(\mu)=\mu -
2\int_0^{ \mu/ d}\bigg(\frac{\sin  \pi v }{\pi v }\bigg)^2 ~dv
$$
and we see, for example, that
$j_1^+(1.94)<1$ so that there must be gaps as large as 1.94 times the average spacing.

\section{The asymptotic large sieve revisited}
Here we include a sketch of the asymptotic large sieve (ALS) results we need. We do this because we regard this current situation as the simplest application of the asymptotic large sieve: the fact that our sequences are related to the M\"{o}bius function and since we are freely assuming GRH, there are no secondary main terms that arise
and this makes the treatment simpler. So, perhaps this treatment will be an instructive first look at the asymptotic large sieve for some readers. Historically, it is the first example the authors considered.

Consider
\begin{equation*}
S=\sum_q \frac{W(q/Q)}{\phi(q)} \sideset{}{^*}\sum_{\chi \bmod q} A(\chi) B(\overline{\chi})
\end{equation*}
where
\begin{equation}
\label{eqn:seqs}
A(\chi)= \sum_{m\le X}
 {a_n\chi(n)}  \qquad
\mathrm{and} \qquad
B(\chi)=
\sum_{n\le X}
 {b_n\chi(n)}  \end{equation}
 and $X\ll Q^{2-\epsilon}$.
We assume that  the bounds
\begin{eqnarray} \label{eqn:bounds}
\sum_{m\le u\atop (m,c)=1} a_{m\ell} \psi(m) \ll  \ell^{-1/2} Q^{\epsilon} \qquad
\mathrm{and} \qquad
\sum_{n\le v\atop (n,c)=1}  b_{n\ell} \psi(b) \ll  \ell^{-1/2} Q^{\epsilon}
\end{eqnarray}
hold uniformly for any $c, \ell, u, v  \ll X $ and any character $\psi$ with conductor $\ll Q$.

We write
 \begin{equation*}S =\sum_{m,n } a_m b_n \Delta(m,n)\end{equation*}
where
\begin{equation*}\Delta(m,n):=\sum_q \frac{W(q/Q)}{\phi(q)} \sideset{}{^*}\sum_{\chi\bmod q}\chi(m)\overline{\chi(n)}.\end{equation*}
 \begin{lem}If $(mn,q)=1$, then
\begin{equation*}\sideset{}{^*}\sum_{\chi\bmod q}\chi(m)\overline{\chi(n)}=
\sum_{ \ d\mid q\atop d\mid (m-n)} \phi(d) \mu(q/d).\end{equation*}
 \end{lem}
Applying Lemma 1 we find that
\begin{equation*}\Delta(m,n)=\sum_{ (cd,mn)=1\atop d\mid m-n} \frac{W(cd/Q)\mu(c)\phi(d)}{\phi(cd)}.\end{equation*}
 \begin{lem} We have
\begin{equation*}\frac{\phi(d)}{\phi(cd)}=\frac{1}{\phi(c)}\sum_{ a\mid c\atop a\mid d} \frac{\mu(a)}{a}.\end{equation*}
\end{lem}
Thus,
\begin{equation*}\Delta(m,n)=\sum_{ (acd,mn)=1\atop ad\mid (m-n)}
 \frac{W(a^2cd/Q)\mu(a)\mu(ac)}{a\phi(ac)}.\end{equation*}
Now we separate the diagonal terms from the non-diagonal ones.
\begin{proposition} \label{proposition:diag} We have
\begin{equation*} \Delta(m,m)=\hat{W}(1)Q\frac{\phi(m)}{m}\prod_p\left(1-\frac1{p^2}-\frac{1}{p^3}\right)
\prod_{p\mid
m}\left(1-\frac1{p^2}-\frac{1}{p^3}\right)^{-1}+O_\epsilon((Qm)^\epsilon)
\end{equation*}
\end{proposition}
\begin{proof}
 We have
\begin{equation*}
\begin{split} \Delta(m,m)&=\sum_{ (acd,m)=1 }\frac{\mu(a)\mu(ac)}
{a\phi(ac)}W\left(\frac{a^2cd}{Q}\right)\\
&=\frac{1}{2\pi i}\int_{(2)}Q^s\hat{W}(s)\zeta(s)\prod_{p\mid
m}(1-\tfrac{1}{p^s})\sum_{ (ac,m)=1 }\frac{\mu(a)\mu(ac)}{a^{1+2s}c^s\phi(ac)}~ds. \end{split}\end{equation*}
 The sums over $a$ and $c$  are absolutely convergent for $\sigma >0$
and $\hat{W}(s)$ is of rapid decay in the vertical direction.
Let $\epsilon>0$. We shift the path of integration to the $\epsilon$-line and pick up the
residue from the pole  of $\zeta(s)$  at $s=1$. Thus
\begin{equation*}
 \Delta(m,m)= \hat{W}(1)Q\frac{\phi(m)}{m}\sum_{ (ac,m)=1 }
\frac{\mu(a)\mu(ac)}{a^3c\phi(ac)}+O\big((Qm)^\epsilon\big).
\end{equation*}
\end{proof}
The sum over $a$ and $c$ in the main term is
\begin{equation*}=\prod_{p\nmid m} \left(1+\frac{1}{p^3(p-1)}-\frac{1}{p(p-1)}\right)
=\prod_{p\nmid m} \left(1-\frac{1}{p^2}-\frac{1}{p^3}\right).\end{equation*}

Now we shall assume that $m\ne n$.
We introduce a parameter $C $ and split the sum over $c$ in $\Delta$
so that we have $\Delta(m,n)=L(m,n)+U(m,n)$
where
\begin{equation*}L(m,n)=\sum_{ (acd,mn)=1\atop ad\mid m-n,c\le C} \frac{W(a^2cd/Q)
\mu(a)\mu(ac)}{a\phi(ac)}\end{equation*}
and
\begin{equation*}U(m,n)=\sum_{ (acd,mn)=1\atop ad\mid m-n,c> C} \frac{W(a^2cd/Q)\mu(a)\mu(ac)}{a\phi(ac)}.\end{equation*}

Let us consider $U$ first.  We replace the condition $ad\mid (m-n)$ by a
sum over all characters modulo $ad$. Thus,
\begin{equation*}U(m,n)=\sum_{ (acd,mn)=1\atop c>C}
\frac{W(a^2cd/Q)\mu(a)\mu(ac)}{a\phi(ac)\phi(ad)}\sum_{\psi \bmod
ad}\psi(m)\overline{\psi}(n).\end{equation*}
 \begin{lem}
   We have
\begin{equation*}
\begin{split} U_E:&=\sum_{ m,n \atop m\ne n} a_{m} b_n
U(m,n)\ll \frac{Q}{C}Q^\epsilon  .
\end{split}
\end{equation*}
 \end{lem}
 \begin{proof}
 We have
 \begin{equation*}U_E=\sum_{ a,c,d\atop c>C}
\frac{\mu(a)\mu(ac)W(a^2cd/Q)}{a\phi(ac)\phi(ad)}
 \sum_{ \psi \bmod ad  } \sum_{m, n\le X\atop m\ne
n,(mn,c)=1}
  a_m \psi(m)
  b_n \overline{\psi(n) }.\end{equation*}
  We include the terms with $m=n$; this introduces an error-term
  of size
  \begin{equation*}\le \sum_{ a,c,d\atop c>C} \frac{ W(a^2cd/Q)}{a\phi(ac) }
  \sum_{ m \le X }
  |a_m||b_m|  \ll \frac{Q}{C}Q^\epsilon,\end{equation*} which is acceptable.

Now, the sum on the right of $U_E$ but with the diagonal terms
$m=n$ included is
\begin{equation*}\ll \sum_{a\le
\sqrt{\frac{2Q}{C}}}\frac{1}{a\phi(a)}\sum_{c>C}\frac{1}{\phi(c)}\sum_{b \le
\frac{2Q}{ac}}\frac{1}{\phi(b)}\sum_{ \psi \bmod b } \big|\sum_{ m\le X\atop (m,c)=1}a_m
\psi(m)\big|\big|\sum_{ n\le X\atop (n,c)=1} b_n
\overline{\psi}(n)\big|.
\end{equation*}
By  (\ref{eqn:bounds}) this is $\ll \frac{Q}{C}(QX)^\epsilon  .$
\end{proof}

 Now we turn to $L(m,n)$.  Let $g=(m,n)$ and
$m=Mg$,
$n=Ng$ so that $(M,N)=1$. We introduce the complementary variable
$e_1$ to complete the product $|M-N|g=|m-n|=ade_1.$  Recall that $m\ne n$ so that
$e_1 > 0$. The goal is to free the variable
 $d$ from the rest of the variables and then eliminate it from the summation. Thus,
\begin{equation*}L(m,n)=\sum_{ (acd,MNg)=1\atop ade_1=|M-N|g,c\le
C}\frac{W(a^2cd/Q)\mu(a)\mu(ac)}{a\phi(ac)}.
\end{equation*}
Now $(ad,g)=1$ implies that $g\mid e_1$, so we replace $e_1$ by $ge$.
Note also that $(M,N)=1$ and $M\equiv N \bmod ad$ together imply that
$(ad,MN)=1$, so we remove that condition from the sum.
Thus,
\begin{equation*}L(m,n)=\sum_{ (c,MNg)=1,c\le C \atop ade=|M-N|, (ad,g)=1
}\frac{W(a^2cd/Q)\mu(a)\mu(ac)}{a\phi(ac)}.
\end{equation*}
Now we express the condition $(d,g)=1$ by the M\"{o}bius formula and
obtain
\begin{equation*}L(m,n)=\sum_{h\mid g}\mu(h)\sum_{ (c,MNg)=1,c\le C\atop adeh=|M-N|, (a,g)=1
}\frac{W(a^2cdh/Q)\mu(a)\mu(ac)}{a\phi(ac)}.
\end{equation*}
At this point, $d$ has been eliminated, since $|M-N|=adeh$ may be expressed as a congruence
$M\equiv N \bmod aeh$.
Note for future reference that $a^2 \le 2Q$.  We introduce
characters $\psi$ modulo $aeh$ to express the condition $M\equiv N
\bmod aeh$; in this way
we obtain
\begin{equation*}L(m,n)=\sum_{h\mid g}\mu(h) \sum_{a,c,e\atop {
(c,MNg)=1\atop c\le C, (a,g)=1}
}\frac{W(ac|M-N|/eQ)\mu(a)\mu(ac)}{a\phi(ac)\phi(aeh)}
\sum_{\psi \bmod aeh}\psi(M)\overline{\psi}(N).
\end{equation*}
  \begin{lem} We have
\begin{equation*}L_E:=\sum_{{  m,n\le X}\atop m\ne n  }^\infty a_m b_n L(m,n)\ll
\frac{X C}{Q} Q^\epsilon.
\end{equation*}
\end{lem}

 \begin{proof} We have
\begin{equation*}L_E=\sum_{g\le X\atop h\mid g } \mu(h)
\sum_{a,c,e,M,N\atop {(M,N)=1, M\ne N,  M,N \le X/g \atop (c,MNg)=1,c\le C,
(a,g)=1} }  a_{Mg} b_{Ng}
\frac{W\big(\frac{ac|M-N|}{eQ}\big)\mu(a)\mu(ac)}{a\phi(ac)\phi(aeh)}
\sum_{ \psi (aeh) }\psi(M)\overline{\psi}(N).
\end{equation*}
We need the variables $M$ and $N$ to be free of each other. To
this end we replace $g$ by $gh$, bring the sum over $M$ and $N$ to
the inside and use the M\"{o}bius formula to eliminate the
condition $(M,N)=1$. Note also that the condition $M\ne N$ is
superfluous since $W(0)=0$. Thus,
\begin{equation*}L_E= \sum_{ a,c,d,e,h\atop c\le
C,(c,gh)=1,(a,g)=1 }\frac{\mu(a)\mu(ac)\mu(d)\mu(h)}{a\phi(ac)\phi(aeh)}\sum_{
\psi (aeh) }|\psi(d)|^2Z\big(X/\ell,\psi,c,\ell,\frac{acd}{eQ}\big)\end{equation*}
where $\ell=gdh$ and
 \begin{eqnarray*} Z(X,\psi,c,\ell,\gamma):=\sum_{{   M,N\le X}\atop (MN,c)=1 }^\infty
 a_{M\ell} b_{N\ell} \psi(M)\overline{\psi(N)}W(\gamma(M-N)).
\end{eqnarray*}
Note  that $ac\ell|M-N|/(eQ)\ge 1$ implies that $e\le
ac\ell|M-N|/Q\le acX/Q$.

Eliminating $d$ from the sum, we have
\begin{eqnarray*}
&L_E\ll \sum_{a}\frac{1}{a\phi(a)}\sum_{c\le C}
\frac{1}{\phi(c)}\sum_{\ell \le X}\sum_{gh\mid \ell} \sum_{e
}\frac{1}{\phi(aeh)}\sum_{ \psi \bmod aeh  }\big|Z\big(X/\ell,\psi,c,\ell,\frac{ac\ell}{gheQ}\big)\big|
\end{eqnarray*}

Now we simplify things for clarity of exposition. We ignore the sums over $a, g, \ell$ and $h$ (i.e. just take all of these variables equal to 1).
We also ignore the coprimality conditions and we treat $\phi(n)$ as $n$ when that is simpler.
 Thus, we consider
 \begin{eqnarray*}
&\mathcal{L}_E :=  \sum_{c\le C}
\frac{1}{c} \sum_{e
}\frac{1}{\phi(e)}\sum_{ \psi \bmod e  }\big|\mathcal{Z}\big(X, \frac{c}{eQ}\big)\big|
\end{eqnarray*}
 where
 \begin{eqnarray*}
 \mathcal{Z}(X, \gamma):=
 \sum_{{   m,n\le X}  }
 a_{m} b_{n} \psi(m)\overline{\psi(n)}W(\gamma(m-n)).
 \end{eqnarray*}
We have suppressed the dependence on $\psi$ since our treatment will be the same for all $\psi$. Let
 $$\Sigma_a(u) =\sum_{m\le u} a_m \psi(m) \qquad \mathrm{and} \qquad \Sigma_b(v) =\sum_{n\le u} b_n \overline{\psi}(n).$$
 As mentioned earlier, GRH implies that $\Sigma_a(u)\ll  Q^\epsilon$ and $\Sigma_b(v)\ll Q^{\epsilon}$
 uniformly for $u,v\ll X$.
 We express the sums over $m$ and $n$ by
Stieltjes integrals and integrate by parts getting
\begin{equation*}
\begin{split} \mathcal{Z}(X,\gamma)&= \gamma \Sigma_a(X)\int_0^X
W'(\gamma(X-v))\Sigma_b(v )~dv-\gamma \Sigma_b(X)\int_0^X
W'(\gamma(u-X))\Sigma_a(u )~du\\
&\qquad \qquad -\gamma^2\int_0^X\int_0^X
W''(\gamma(u-v))\Sigma_a(u)\Sigma_b(v)~du~dv
\end{split}
\end{equation*}
(recall that $W$ is supported on [1,2]).
 Thus,
 \begin{eqnarray*}
\mathcal{L}_E &\ll& Q^\epsilon \sum_{c\le C}
\frac{1}{c} \sum_{e\ll Xc/Q
} \bigg( \int_0^X \frac{c}{eQ}
|W'\big(\frac{c(X-v)}{eQ}\big)| ~dv +\int_0^X\int_0^X\frac{c^2}{e^2Q^2}
|W''\big(\frac{c(u-v)}{eQ})| ~du~dv\bigg)\\
&\ll& \frac{XC}{Q} Q^\epsilon.
\end{eqnarray*}
 Summarizing,  we have
\begin{eqnarray*}
S &=&\hat{W}(1)\prod_p\left(1-\frac1{p^2}-\frac{1}{p^3}\right)Q\sum_{m\le X} a_m b_m\frac{\phi(m)}{m}
\prod_{p\mid
m}\left(1-\frac1{p^2}-\frac{1}{p^3}\right)^{-1} +O\big(Q^\epsilon\big(\frac{XC}{Q}+\frac{Q}{C}\big)\big).
\end{eqnarray*}
If $X=Q^{2-4\epsilon}$ and we choose $C=Q^{2\epsilon}$ then the error term here is $(Q^{1-\epsilon})$ which is smaller than the main term.

To give a completely explicit example, we have
\begin{corollary} Assume the Generalized Riemann Hypothesis. Let $W$ be a $C^\infty$ function supported on $[1,2]$. Then, for any $Q, X \ge 1$ and any $\epsilon >0$  we have
\begin{eqnarray*}&&
\sum_q \frac{W(q/Q)}{\phi(q)} \sideset{}{^*}\sum_{\chi \bmod q}\left|\sum_{m\le X}\frac{\mu(m)\chi(m)}{\sqrt{m}}\right|^2   \\
&&  \qquad \qquad \qquad =
\hat{W}(1)\prod_p\left(1-\frac1{p^2}-\frac{1}{p^3}\right)Q\sum_{m\le X}\frac{\mu^2(m)\phi(m)}{m^2}
\prod_{p\mid
m}\left(1-\frac1{p^2}-\frac{1}{p^3}\right)^{-1}\\
&&\qquad \qquad  \qquad \qquad \qquad +O_\epsilon \big(Q^\epsilon \sqrt{X}\big).
\end{eqnarray*}
\end{corollary}
In the application to $\mathcal{M}(\alpha)$ our $B(\chi)$ term is an infinite series. We truncate the series at $X$ and deal with the small terms as
above. For the larger terms  we have
\begin{equation*}
S=\sum_q \frac{W(q/Q)}{\phi(q)} \sideset{}{^*}\sum_{\chi \bmod q} A(\chi) B(\overline{\chi})
\end{equation*}
but now
\begin{equation}
\label{eqn:seqs1}
A(\chi)= \sum_{m\le X}
 {a_n\chi(n)}  \qquad
\mathrm{and} \qquad
B(\chi)=
\sum_{n> X}
 {b_n\chi(n)}  \end{equation}
 with $X\ll Q^{2-\epsilon}$.
Now the bounds we have are
\begin{eqnarray} \label{eqn:bounds1}
\sum_{m\le u\atop (m,c)=1} a_{m\ell} \psi(m) \ll   Q^{\epsilon} u^{1/2}\qquad
\mathrm{and} \qquad
\sum_{n > v\atop (n,c)=1}  b_{n\ell} \psi(b) \ll  \ell^{-1} Q^{\epsilon} v^{-1/2}
\end{eqnarray}
uniformly for any $c, \ell, u \ll X, $ any character $\psi$ with conductor $\ll Q$, and $v\gg X$.
\begin{thm}
With the assumptions (\ref{eqn:seqs1}) and (\ref{eqn:bounds1}) above, we have
$$S\ll_\epsilon Q^{1-\epsilon}.$$
\end{thm}
The proof is similar.

\end{proof}


\begin{thebibliography}{HD}




\normalsize \baselineskip=17pt





\bibitem[Bre1]{journal} Bredberg, J. On large gaps between consecutive zeros, on the critical line, of some Dirichlet $L$-functions,
arXiv:1003.2290.

\bibitem[Bre2]{journal} Bredberg, J. Large gaps between consecutive zeros, on the critical line, of the Riemann zeta-function, arXiv:1101.3197

\bibitem[CGG]{journal} Conrey, J. B.; Ghosh, A.; Gonek, S. M.
A note on gaps between zeros of the zeta function.
Bull. London Math. Soc.  16  (1984),  no. 4, 421--424.


\bibitem[CI]{journal} Conrey, B.; Iwaniec, H. Spacing of zeros
of Hecke $L$-functions and the class number problem.
 Acta Arith.  103  (2002),  no. 3, 259--312.

\bibitem[CIS]{journal} Conrey, B.; Iwaniec, H.; Soundararajan, K. The asymptotic large sieve, preprint, arXiv:1105.1176.

\bibitem[CIS2]{journal} Conrey, B.; Iwaniec, H.; Soundararajan, K. The sixth moment of Dirichlet L-functions, arXiv:0710.5176

\bibitem[FW]{journal} Feng,S.; Wu, X.   On gaps between zeros of the Riemann zeta function, arXiv:1003.0752.

\bibitem[Ha]{journal} Hall, R. R. On the zeros of the Riemann zeta-function.  J. London Math. Soc. (2)  59  (1999),  no. 1, 65�75.

\bibitem[Mo]{journal} Montgomery, H. L. The pair correlation of zeros of the zeta function.  Analytic number theory (Proc. Sympos. Pure Math., Vol. XXIV, St. Louis Univ., St. Louis, Mo., 1972),  pp. 181�193. Amer. Math. Soc., Providence, R.I., 1973.

\bibitem[MoOd]{chapter} Montgomery, H. L.; Odlyzko, A. M. Gaps
between zeros of the zeta function.
Topics in classical number theory, Vol. I, II (Budapest, 1981),  1079--1106,
Colloq. Math. Soc. J\'{a}nos Bolyai, 34, North-Holland, Amsterdam, 1984.

\bibitem[M]{journal} Mueller, J. On the difference between consecutive zeros of the Riemann zeta function.  J. Number Theory  14  (1982), no. 3, 327�331.

\bibitem[RS]{journal}  Rudnick, Z.; Sarnak, P. Zeros of principal $L$-functions and random matrix theory. A celebration of John F. Nash, Jr.  Duke Math. J.  81  (1996),  no. 2, 269�322.

\bibitem[St]{journal} Stark, H. M. On the zeros of Epstein's zeta function. Mathematika 14 1967 47–55.
\end{thebibliography}
\end{document}